\documentclass{article}

\usepackage{verbatim}
\usepackage{graphics}
\usepackage[pdftex]{graphicx}
\usepackage{sidecap}
\usepackage{amsmath, amssymb, amsfonts, amsthm}
\usepackage{url}
\usepackage{setspace}
\usepackage{algorithmic}
\usepackage{enumitem}
\setlist{nolistsep}

\def \PG[#1,#2]{PG(#1,#2)}
\def \AG[#1,#2]{AG(#1,#2)}

\newtheorem{theorem}{Theorem}
\newtheorem{conjecture}[theorem]{Conjecture}
\newtheorem{lemma}[theorem]{Lemma}

\newtheorem{proposition}[theorem]{Proposition}

\newcommand{\RR}{\ensuremath{\mathbb R}}
\newcommand{\pts}{\mathcal P}

\newcommand{\vrts}{\mathcal S}

\newcommand{\surfs}{\mathcal S}

\newcommand{\bi}[1]{{B}(#1)}
\newcommand{\bis}{\mathcal B}

\newcommand{\en}{\mathcal Q}

\newcommand{\eps}{\varepsilon}

\begin{document}

\title{A refined energy bound for distinct perpendicular bisectors}
\author{Ben Lund \thanks{This work was supported by ERC grant 267165 DISCONV}}
\maketitle

\begin{abstract}
Let $\pts$ be a set of $n$ points in the Euclidean plane.
We prove that, for any $\eps > 0$, either a single line or circle contains $n/2$ points of $\pts$, or the number of distinct perpendicular bisectors determined by pairs of points in $\pts$ is $\Omega(n^{52/35 - \eps})$, where the constant implied by the $\Omega$ notation depends on $\eps$.
This is progress toward a conjecture of Lund, Sheffer, and de Zeeuw, that either a single line or circle contains $n/2$ points of $\pts$, or the number of distinct perpendicular bisectors is $\Omega(n^2)$.

The proof relies bounding the size of a carefully selected subset of the quadruples $(a,b,c,d) \in \pts^4$ such that the perpendicular bisector of $a$ and $b$ is the same as the perpendicular bisector of $c$ and $d$.
\end{abstract}

\section{Introduction}

Many classic problems in discrete geometry ask for the minimum number of distinct equivalence classes of subsets of a fixed set of points under some geometrically defined equivalence relation.
The seminal example is the Erd\H{o}s distinct distance problem \cite{erdos1946sets}: How few distinct distances can be determined by a set of $n$ points in the Euclidean plane?
Guth and Katz have nearly resolved the Erd\H{o}s distinct distance question \cite{guth2015erdos}, but many questions of this type remain wide open.

In this paper, we investigate the question: How few distinct perpendicular bisectors can be determined by a set of $n$ points in the Euclidean plane?
Distinct perpendicular bisectors were previously investigated by the author, Sheffer, and de Zeeuw \cite{lund2015bisector}, and a finite field analog was studied by Hanson, the author, and Roche-Newton \cite{hanson2016distinct}.

Without any additional assumption, it is not too hard to give a complete answer to this question.
The vertices of a regular $n$-gon determine $n$ distinct perpendicular bisectors.
Each point of an arbitrary point set $\pts$ of $n$ points determines $n-1$ distinct bisectors with the remaining points of $\pts$, so the only question is whether it is possible that $n$ points determine only $n-1$ bisectors.
A set of two points determines only a single perpendicular bisector, and in subsection \ref{sec:bisectorsAtLeastN}, we give a simple geometric argument showing that $n$ points determine at least $n$ bisectors when $n>2$.

Assume that $\pts$ is a set of $n$ points such that no circle or line contains more than $K$ points of $\pts$.

The author, Sheffer, and de Zeeuw \cite{lund2015bisector} proved the following lower bound on $|\bis|$, the number of distinct bisectors determined by $\pts$.
For any $\eps > 0$,
\begin{equation}\label{eq:weakBound}|\bis| = \Omega\left(\min\left\{K^{-{2}/{5}}n^{{8}/{5}-\eps},
K^{-1} n^2\right\}\right),\end{equation}
where the implied constant depends on $\eps$.
The same paper proposes the following conjecture.
\begin{conjecture}\label{conj:distinctBisectors}
For any $\delta >0$, there is a constant $c>0$ depending on $\delta$ such that either a single line or circle contains $(1 - \delta)n$ points of $P$, or $|\bis| \geq c n^2$.
\end{conjecture}

In this paper, we prove the following.
\begin{theorem}\label{thm:distinctBisectors}
For any $\delta,\eps > 0$, there is a constant $c > 0$ depending on $\delta$ and $\eps$ such that either a single circle or line contains $(1-\delta )n$ points of $\pts$, or
$|\bis| \geq c n^{52/35 - \eps})$.
\end{theorem}
This improves on the earlier result (\ref{eq:weakBound}) of the author, Sheffer, and de Zeeuw in the case that some line or circle contains $\Omega(n^{2/7+\eps})$ points of $\pts$, and gives the first non-trivial result on Conjecture \ref{conj:distinctBisectors} in the case that a single line or circle contains a constant fraction of the points of $\pts$.

The proof (in \cite{lund2015bisector}) of inequality (\ref{eq:weakBound}) uses the, now standard, method of bounding the ``energy"\footnote{The term \textit{additive energy}, referring to the number of quadruples $(a,b,c,d)$ in some underlying set of numbers such that $a+b = c+d$, was coined by Tao and Vu \cite{taoadditive}. Starting with the work of Elekes and Sharir \cite{elekes2011incidences}, and Guth and Katz \cite{guth2015erdos} on the distinct distance problem, the strategy of using geometric incidence bounds to obtain upper bounds on analogously defined energies has become indispensable in the study of questions about the number of distinct equivalent subsets.} of the quantity in question.
Specifically, we write $\bi{a,b}$ for the perpendicular bisector of distinct points $a,b$, and define the \textit{bisector energy} to be the size of the set
$$\en = \{(a,b,c,d) \in \pts^4:a \neq b, c\neq d, \bi{a,b} = \bi{c,d}\}.$$
It is easy to see that $|\en| \leq n^2(n-1)$, since each element of $\en$ is determined by $(a,b,c)$.
Taking $\pts$ to be the vertices of a regular $n$-gon shows that this bound is tight.
In \cite{lund2015bisector}, it is shown that
\begin{equation}\label{eq:energyBound}|\en| = O\left(K^{{2}/{5}}n^{{12}/{5}+\eps} + Kn^2\right),\end{equation}
where $K$ is the largest number of points of $\pts$ contained in any line or circle.
The same paper includes the conjecture that the strongest possible bound is $|\en| = O(Kn^2)$.
By the Cauchy-Schwarz inequality (see, for example, the proof of Lemma \ref{thm:applCS}),
$$|\bis| \geq n^2(n-1)^2/|\en|.$$
From here, a simple substitution gives (\ref{eq:weakBound}).

Following this argument, even a tight bound of $|\en| = O(Kn^2)$ would only give $|\bis| \geq \Omega(n^2K^{-1})$.
This only meets the bound of Conjecture \ref{conj:distinctBisectors} when $K$ is a constant not depending on $n$, and does not give any non-trivial bound for $K = \Omega(n)$.
Hence, it initially seems hopeless to use an energy bound to make substantial progress toward Conjecture \ref{conj:distinctBisectors} in the case that a single line or circle contains many points of $\pts$.

The main new idea in this paper is to apply an energy bound to a refined subset of the pairs of points of $\pts$.
We show that there is a large set $\Pi \subset \pts \times \pts$ of pairs of points, such that
$$\en^* =  \{(a,b,c,d) \in \pts^4: (a,b), (c,d) \in \Pi, \bi{a,b} = \bi{c,d}\}$$
is small.
In particular, we define $\Pi$ to be the set of pairs of points of $\pts$ that are not contained in any circle or line that contains too many points of $\pts$.
We use a point-circle incidence bound, proved in \cite{aronov2002cutting, agarwal2004lenses, marcus2006intersection}, to show that $\Pi$ must be large.
We use a slightly modified version of the argument used to bound $\en$ in \cite{lund2015bisector} to show that $\en^*$ must be small.
An application of the Cauchy-Schwarz inequality then shows that there must be many distinct bisectors determined by pairs of points in $\Pi$, which of course implies that there must be many bisectors in total.

\section{Proofs}

Throughout this section, $\pts$ is a set of $n$ points in the plane, and $\bis$ is the set of distinct perpendicular bisectors determined by the pairs of points of $\pts$.
For any two distinct points $a,b$, we use $B(a,b)$ to denote the perpendicular bisector of $a$ and $b$, and use $|ab|$ to denote the distance between $a$ and $b$.

We rely on the connection between perpendicular bisectors and reflections.
This is that, for any two distinct points $a,b$ in the plane, $b$ is the reflection of $a$ over $B(a,b)$.

\subsection{There are at least $n$ bisectors} \label{sec:bisectorsAtLeastN}
We give the best possible general lower bound on $|\bis|$.
\begin{proposition}
	If $n > 2$, then $|\bis| \geq n.$
\end{proposition}
\begin{proof}
	Since any point $a \in \pts$ determines $n-1$ distinct bisectors with the remaining points $\pts \setminus \{a\}$, it is sufficient to show that there are three points $a,b,c$ such that $\bi{b,c}$ is distinct from $\bi{a,x}$ for any $x \in \pts$.
	
	Suppose that there is a line $\ell$ containing at least $3$ points of $\pts$.
	Order the points along $\ell$, let the first two points of $\pts \cap \ell$ be $b,c$, and let $a$ be any other point of $\pts \cap \ell$.
	The point $x$ such that $B(a,x) = B(b,c)$ lies in $\ell$ and precedes $b,c$, and hence can't be in $\pts$.
	
	Now suppose that no $3$ points are collinear.
	Let $a,b \in \pts$ so that $|ab|$ is minimal, and let $c \in \pts$ so that the angle $\angle abc$ is minimal.
	If $a$ is on the same side of $\bi{b,c}$ as $c$, then $|ac| \leq |ab|$, which is a contradiction.
	If $a$ is on the line $\bi{b,c}$, then there is no point $x$ such that $\bi{a,x} = \bi{b,c}$, and we have accomplished our goal.
	Hence, we may suppose that $a$ and $b$ are on the same side of $\bi{b,c}$ -- see Figure \ref{fig:Prop4a}.
	
	Let $x$ be the reflection of $a$ over $\bi{b,c}$. Since $x$ is in the interior of the cone defined by $\angle abc$, we have that $\angle abx$ is less than $\angle abc$ - see Figure \ref{fig:Prop4}.
	Since $c$ was chosen so that $\angle abc$ is minimal, $x \notin \pts$, which completes the proof.
\end{proof}

\begin{figure}
	\begin{center}
	\includegraphics{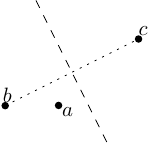}
	\caption{Since $|ab| < |bc|$, both $a$ and $b$ must lie on the same side of $B(b,c)$. }
	\label{fig:Prop4a}
	\end{center}
\end{figure}

\begin{figure}
	\begin{center}
	\includegraphics{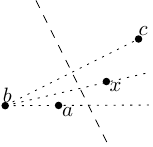}
	\caption{If $B(a,x) = B(b,c)$, and $a$ and $c$ are on opposite sides of $\overline{ax}$, then $\angle abx < \angle abc$. }
	\label{fig:Prop4}
	\end{center}
\end{figure}

\subsection{Proof of Theorem \ref{thm:distinctBisectors}} \label{sec:distinctBisectors}

The proof of Theorem \ref{thm:distinctBisectors} has three main parts.
First, we show that, for any $\eps>0$, if a single circle or line contains at least $\eps n$ points of $\pts$, then $\Omega(n^2)$ distinct bisectors are determined, with the implied constant depending on $\eps$.
Second, we show that, for suitably chosen constants $c_1,c_2 > 0$, if no circle or line contains at least $c_2 n$ points of $\pts$, then there is a set $\Pi$ of $\Omega(n^2)$ pairs of points such that no pair of points in $\Pi$ is contained in any circle that contains more than $M=c_1 n^{2/7} \log^{2/7}n$ points of $\pts$.
Third, we apply an energy argument (as in \cite{lund2015bisector}) to show that, for any $\eps > 0$, there are $O(M^{2/5}n^{12/5+\eps} + M n^2)$ quadruples $(a,b,c,d)$ of points in $\pts$ such that $B(a,b) = B(c,d)$ and $(a,b),(c,d) \in \Pi$, with the implied constant depending on $\eps$.
From there, a straightforward application of the Cauchy-Schwarz inequality finishes the proof.

\noindent{\bf Handling heavy circles.}
First, a geometric lemma.

\begin{lemma}\label{thm:sharedBisectors}
	Let $C$ be a circle or a line, and let $p,q \notin C$ with $p\neq q$.
	Then,
	$$|\{(r,s) \in C \times C : \bi{p,r} = \bi{q,s}\}| \leq 2.$$
\end{lemma}
\begin{proof}
	If $r,s \in C$ such that $B(p,r) = B(q,s)$,
	then $p,q$ is contained in the reflection of $C$ over $B(p,r)$ - see Figure \ref{fig:lemma5}.
	Since $p\neq q$, there are at most two reflections of $C$ that contain $(p,q)$ if $C$ is a circle, and at most one such reflection of $C$ if $C$ is a line.
	Since we can recover $r,s$ uniquely given one of these reflections, there are at most two choices for the pair $(r,s)$.

\end{proof}

\begin{SCfigure}\label{fig:lemma5}
	\caption{Illustration for Lemma \ref{thm:sharedBisectors}. Both choices of $r,s$ are shown for the given $p,q$ and $C$. The dashed circle is the reflection of $C$ over the dashed line $B(p,r) = B(q,s)$, and the dotted circle is the reflection of $C$ over the dotted line $B(p,r') = B(q,s')$. \vspace{15mm}}
	\includegraphics{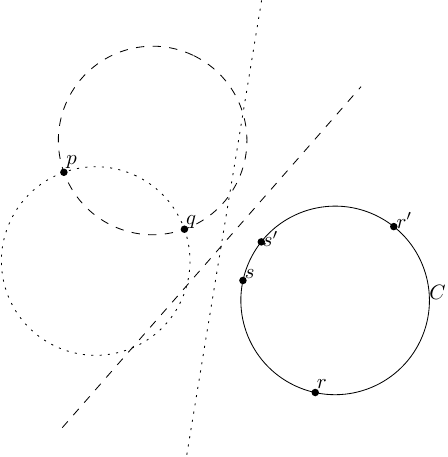}
\end{SCfigure}

Combining Lemma \ref{thm:sharedBisectors} with a combinatorial argument gives the result.

\begin{lemma}\label{thm:heavyCircles}
For any $\eps > 0$, if a single line or circle contains exactly $\eps n$ points of $\pts$, then
$$|\bis| >  \min(\eps/2, 1-\eps) \cdot \eps n^2/2.$$
\end{lemma}

\begin{proof}
Let $C$ be a circle or line that contains $\eps n$ points of $\pts$.
Let $\pts' \subset \pts$ be a set of $k = \min(\eps/2, 1-\eps)n$ points that are not in $C$.
Let $p_1, p_2, \dots p_k$ be an arbitrary ordering of the points of $\pts'$.
Then, by Lemma \ref{thm:sharedBisectors}, $p_i$ determines a set $\bis(p_i)$ of at least $\eps n - 2(i-1)$ distinct perpendicular bisectors with the points of $\pts \cap C$, such that no element of $\bis(p_i)$ is an element of $\bis(p_j)$ for any $j<i$.
Summing over $i$, we have
$$\sum_{i =1}^{k}|\bis(p_i)| \geq \sum_{i =1}^{ k} (\eps n - 2(i-1)) = \eps k n - k^2 + k > \eps n k/2,$$
which proves the lemma.
\end{proof}

We will apply Lemma \ref{thm:heavyCircles} for some $\eps < 2/3$.
In this range, the bound is $|\bis| > \eps^2n^2/4$.

\noindent{\bf Refining the pairs of points.}
Let $c_1, c_2 > 0$, with specific values to be fixed in the proof of Lemma \ref{th:PiExists}.
Let $\Pi \subseteq \pts \times \pts$ be the set of ordered pairs of distinct points of $\pts$ such that no pair in $\Pi$ is contained in a line or a circle that contains more than $c_1 n^{2/7}\log^{2/7}n$ points of $\pts$.
We use a point-circle incidence bound to show that, under the assumption that no circle contains $c_2 n$ points, $|\Pi| = \Omega(n^2)$.

Denote by $s_k$ the number of lines and circles that each contain at least $k$ points of $\pts$, denote by $s_{=k}$ the number of lines and circles that each contain exactly $k$ points of $\pts$.

The strongest known point-circle incidence bound is derived using a combination of the papers by Agarwal, Nevo, Pach, Pinchasi, Sharir, and Smorodinsky \cite{agarwal2004lenses} and by Marcos and Tardos \cite{marcus2006intersection}.
A slightly weaker bound was proved earlier by Aronov and Sharir \cite{aronov2002cutting}.
Combining the strongest known point-circle incidence bound with the point-line incidence bound proved by Szemer\'edi and Trotter \cite{szemeredi1982extremal} gives the following bound on $s_k$.

\begin{lemma}\label{thm:pointCircleBound}
$$s_k = O(n^{3} k^{-11/2} \log n + n^2 k ^{-3} + nk^{-1}).$$
\end{lemma}

Recall that the definition of $\Pi$ depends on $c_1$, which we fix in the next Lemma.

\begin{lemma}\label{th:PiExists}
There are constants $c_1,c_2 > 0$ such that either a single line or circle contains $c_2 n$ points of $\pts$, or $|\Pi| = \Omega(n^2)$.
\end{lemma}
\begin{proof}
	Let $M = c_1 n^{2/7}\log^{2/7}n$, and let $U = c_2 n$.
	Assume that no circle or line contains $U$ points of $\pts$, since otherwise we're done.
	
	Let $T$ be the number of triples $(p,q,C)$ of two points $p,q$ and a line or circle $C$ such that $p,q \in C \cap \pts$ and $M \leq |C \cap \pts| \leq U$.
	Then,
\[ T = \sum_{k = M}^U k^2 s_{=k} = \sum_{k = M}^U k^2(s_k - s_{k+1}) \leq \sum_{k = M+1}^U 2k s_k + M^2 s_M.\]

We use Lemma \ref{thm:pointCircleBound} to bound $s_k$.
Since the term $n^3 k^{-11/2} \log n$ is dominant for $k=M$, this gives
\[ T \leq O \left( \sum_{k \geq M} n^{3}k^{-9/2}\log n +  \sum_{k \geq M} n^2 k^{-2} + \sum_{k \leq U} n + n^3M^{-7/2}\log n \right). \]

If we take $c_1$ to be sufficiently large depending on the implied constant in Lemma \ref{thm:pointCircleBound}, then the first and fourth terms will each be bounded by $n^2/10$, and the second term will be bounded by $O(n^{9/7}\log^{-5/7}n)$.
If we take $c_2$ to be sufficiently small, again depending on the implied constant in Lemma \ref{thm:pointCircleBound}, then the third term will be bounded by $n^2/10$.
Adding these contributions together shows that $T < n^2/2$.

Since $T$ counts each pair of points that is contained in a circle that contains at least $M$ points of $\pts$ (possibly more than once), $|\Pi| \geq n(n-1) - T = \Omega(n^2)$.
\end{proof}

\noindent{\bf Bounding the energy.}
Next, we bound a refinement of the bisector energy depending on $\Pi$.
Our argument is essentially identical to proof of the the analogous bound in \cite{lund2015bisector}, and in fact we refer to \cite{lund2015bisector} for many of the key facts used.

Let $\pts^{2*} \subset \pts^{2}$ be the set of pairs of distinct points of $\pts$.
For each pair $(a,b) \in P^{2*}$, let $C(a,b)$ be the maximum number of points on any circle or line that contains $a,b$.
Let 
\begin{align*}
\Pi_K &= \{(a,b) \in P^2 : a \neq b, \, C(a,b) \leq K\}, \\
\mathcal{Q}_K &= \{(a,b,c,d) \in P^4 : (a,b),(c,d) \in \Pi_K, \, \bi{a,b} = \bi{c,d}\}.
\end{align*}

Our goal is to prove an upper bound on $\mathcal{Q}_K$.
Note that, if $B(a,b) = B(c,d)$, then
\begin{align}
(a+b-c-d) \cdot (a-b) &= 0, \label{eq:firstBisectorConstraint}\\
(a+b-c-d) \cdot (c-d) &=0. \label{eq:secondBisectorConstraint}
\end{align}
Indeed, $(a+b-c-d)$ is parallel to the line through the midpoints of $(a,b)$ and $(c,d)$, hence (\ref{eq:firstBisectorConstraint}) requires that this line be perpendicular to the line through $a$ and $b$, and (\ref{eq:secondBisectorConstraint}) requires that this line be perpendicular to the line through $c$ and $d$.
Hence, any quadruple that contributes to $\mathcal{Q}$ satisfies (\ref{eq:firstBisectorConstraint}) and (\ref{eq:secondBisectorConstraint}).
There are some quadruples $(a,b,c,d)$ satisfy (\ref{eq:firstBisectorConstraint}) and (\ref{eq:secondBisectorConstraint}), but do not contribute to $\mathcal{Q}$.
However, these cases only occur if $a=b$ or $c=d$; see \cite[Lemma 3.1]{lund2015bisector}.

For any pair $(a,c)$ of distinct points, let $S_{ac}$ be the set of pairs $(b,d)$ satisfying (\ref{eq:firstBisectorConstraint}) and (\ref{eq:secondBisectorConstraint}).
Let $G$ be the incidence graph between varieties $\{S_{ac}:(a,c) \in P^{2*}\}$ and pairs of points $(b,d) \in P^{2*}$.
Let $G_K$ be the subgraph of $G$ so that an edge $(S_{ac}, (b,d)) \in G$ is in $G_K$ if and only if $(a,b),(c,d) \in \Pi_K$.
Note that the edges of $G_K$ correspond exactly to the elements of $\en_K$.

In \cite{lund2015bisector}, the authors used a geometric incidence argument to bound the number of edges in $G$.
We will use the same argument to bound the number of edges in $G_K$; the only difficulty is in identifying the property of $G_K$ that enables us to run the argument of \cite{lund2015bisector}.

We need a couple of geometric facts from \cite{lund2015bisector}.
First, if $(S_{ac}, (b,d)) \in G$, then $|ac| = |bd|$.
This is easy to derive from equations (\ref{eq:firstBisectorConstraint}) and (\ref{eq:secondBisectorConstraint}): expand the products and subtract (\ref{eq:secondBisectorConstraint}) from (\ref{eq:firstBisectorConstraint}).
Also see Figure \ref{fig:bisectorQuadruple}.

\begin{figure}
	\begin{center}
	\includegraphics{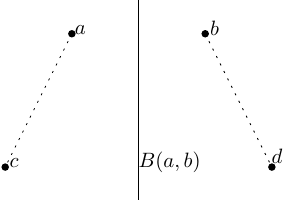}
	\caption{Since $B(a,b) = B(c,d)$, the pair $(b,d)$ is the reflection of $(a,c)$ over the line $B(a,b)$, and so $|ac| = |bd|$. }
	\label{fig:bisectorQuadruple}
	\end{center}
\end{figure}

Next, we have the following lemma, which gives a constraint on pairs $(b,d)$ such that $B(a,b) = B(c,d)$ and $B(a',b) = B(c',d)$ for fixed points $a,c,a',c'$.

\begin{lemma}[Lemma 3.2 \cite{lund2015bisector}]\label{lem:intersection}
	Let $a,c,a',c' \in \mathbb{R}^2$ such that $(a,c)\neq (a',c')$, and $|ac| = |a'c'| = \delta \neq0$.
	There exist curves $C_1,C_2\subset \RR^2$ depending on $a,c,a',c'$, which are either two concentric circles or two parallel lines, such that, if $B(a,b) = B(c,d)$ and $B(a',b) = B(c',d)$, then $a,a',b \in C_1$ and $c,c',d \in C_1$.
\end{lemma}

Now we can establish the property we need from $G_K$.

\begin{lemma}\label{th:Kcut}
	If ${S}_{ac}$ and ${S}_{a'c'}$ more than $K+2$ common neighbors in $G$, then they have no common neighbors in $G_K$. 
\end{lemma}

\begin{proof}
	We first show that, if $c \neq c'$, then $S_{ac}$ and $S_{ac'}$ have no common neighbors in $G_K$.
	Let $(b,d) \in S_{ac}$.
	If $b \neq a$ and $d \neq c$, then $a,b,c,d$ are the vertices of an isoceles trapzoid.
	Hence, $c$ is determined uniquely by $a,b,d$, and $(b,d) \notin S_{ac'}$
	If $b = a$, then $(a,b) \notin \Pi_K$, so $(S_{ac},(b,d)) \notin G_K$, and similarly for the case $d=c$.
	By the same reasoning, $S_{ac}$ and $S_{a'c}$ have no common neighbors in $G_K$ for $a \neq a'$.
	
	Assume that $a \neq a'$ and $c \neq c'$, and let $(b,d) \in S_{ac} \cap S_{a'c'}$.
	By Lemma \ref{lem:intersection}, if $d \neq c$ and $d \neq c'$, then $b$ is contained in a line or circle $C_1$ depending on $a,a',c,c'$.
	If $d = c$, then the location of $b$ is uniquely determined by $a',c',c$, and similarly if $d = c'$.
	Hence, if $S_{ac} \cap S_{a'c'} > K+2$, then $C_1$ contains more than $K$ points of $\pts$.
	In this case, none of the pairs $(a,b)$ with $a,b \in C_1$ are in $\Pi_K$, so the corresponding edges are missing in $G_K$.
\end{proof}

The following is the general incidence bound we use to control the size of $G_K$.
It is a slight generalization of a bound in \cite{lund2015bisector}, which is in turn a generalization of a bound in \cite{fox2014semi}.
See \cite{fox2014semi} for definitions of the algebraic terms used.

\begin{theorem}\label{th:incidencebound}
	Let $\vrts$ be a set of $n$ constant-degree varieties, and let $\pts$ be a set of $m$ points, both in $\RR^d$, where $d \geq 2$.
	Let $s \geq 2$ be a constant, and $t \geq 2$ be a function of $m,n$.
	Let $G$ be the incidence graph of $\pts \times \vrts$.
	Let $G' \subseteq G$ such that, if a set $L$ of $s$ left vertices has a common neighborhood of size $t$ or more in $G$, then no pair of vertices in $L$ has a common neighbor in $G'$.
	Moreover, suppose that $\pts\subset V$, where $V$ is an irreducible constant-degree variety of dimension $e$. 
	Then, for any $\eps > 0$,
	\[|G'| = O\left(m^{\frac{s(e-1)}{es-1}+\eps}n^{\frac{e(s-1)}{es-1}}t^{\frac{e-1}{es-1}} +tm+n\right),\]
	with the implied constant depending on $\eps$.
\end{theorem}

The proof of Theorem \ref{th:incidencebound} is nearly identical to the proof of \cite[Theorem 2.5]{lund2015bisector}.
Instead of reproducing the full proof here, we briefly state how to modify the proof of \cite[Theorem 2.5]{lund2015bisector} to obtain the more general Theorem \ref{th:incidencebound}; the remainder of this paragraph is refers to the notation from the proof of \cite[Theorem 2.5]{lund2015bisector}.
We partition $G_K$ into $I_1, I_2,$ and $I_3$ as in the proof of \cite[Theorem 2.5]{lund2015bisector}.
Bounding $|I_2|$ and $|I_3|$ requires no change at all; these bounds only depend on the fact that $G_K$ is $K_{s,t}$-free.
Any incidence in $I_1$ occurs in some irreducible component $W$ of $V \cap Z(f)$, where $Z(f)$ is the zero set of our partitioning polynomial, such that $W$ is fully contained in some variety $S \in \surfs$.
In bounding $I_1$, we need to make use of the observation that if  there are at least $t$ varieties of $\surfs$ that fully contain $W$, then, by Lemma \ref{th:Kcut}, no pair of vertices corresponding to the points contained in $W$ has a common neighbor in $G_K$ among the varieties that contain $W$.

We now have all of the tools in place to bound $|\en_K|$.

\begin{lemma}\label{thm:QK}
	For any $2 \leq K \leq n$ and $\eps > 0$,
	$$|\en_K| = O\left(K^{{2}/{5}}n^{{12}/{5}+\eps} + K n^2 \right),$$
	with the implied constant depending on $\eps$.
\end{lemma}
\begin{proof}
	Let $\delta_1, \ldots, \delta_D$ denote the distinct non-zero distances determined by pairs of distinct points in $\pts$.
	Let
	\begin{align*}
	\pts^2_i &= \{(b,d) \in \pts^{2*} : |bd| = \delta_i\},\\
	\surfs_i &= \{S_{ac} \in \surfs : |ac| = \delta_i\}, \\
	G'_i &= \{(S_{ac},(b,d)) \in G_K : |ac| = |bd| = \delta_i\}.
	\end{align*}
	Let $$m_i = |\pts^2_i| = |\surfs_i|.$$
	
	As observed above, each quadruple $(a,b,c,d) \in Q$ satisfies $|ac| = |bd|$.
	Hence, it suffices to study each $G'_i$ separately.
	That is, we have
	$$|Q_K| = |G_K| = \sum_{i=1}^D |G'_i|.$$
	
	For each $i \in [D]$, we apply Theorem \ref{th:incidencebound} with: $m=n=m_i$, $d=4$, $\surfs = \surfs_i$, $\pts = \pts^2_i$, $s = 2$, $t=K+3$, and $V = \{(a,b):|ab| = \delta_i \}$, and $e=3$.
	This gives, for any $\eps > 0$,
	\begin{equation}\label{eq:Inc-i}|G'_i| = O(K^{2/5}m_i^{7/5+\eps} + Km_i),\end{equation}
	with the implied constant depending on $\eps$.
	
	Let $J$ be the set of indexes $1\le j\le D$ for which the bound in \eqref{eq:Inc-i} is dominated by the term $K^{{2}/{5}} m_j^{{7}/{5}+\eps}$. By recalling that $\sum_{j=1}^D m_j = n(n-1)$, we get
	\[\sum_{j\not\in J}|G_j'| = O\left(Kn^2\right).\]
	Next we consider $\sum_{j\in J}|G'_j| = \sum_{j\in J} O(K^{2/5} m_j^{7/5+\eps} )$.
	By H\"older's inequality,
	\[\sum_{j \in J} m_j^{7/5} = \sum_{j \in J} m_j^{3/5}(m_j^2)^{2/5} \leq \left( \sum_{j \in J} m_j \right)^{3/5} \left(\sum_{j \in J}m_j^2 \right)^{2/5}.  \]
	
	Guth and Katz \cite[Proposition 2.2]{guth2015erdos} proved a tight bound on $\sum m_j^2$:
	\[\sum m_j^2 = O(n^3\log n).\]
	
	Combining these estimates, for any $\eps > 0$,
	\[\sum_{j \in J} |G_j'| = O_\eps(K^{2/5}n^{12/5 + \eps}).\]
\end{proof}

\noindent \textbf{Finishing the proof.}
Lemma \ref{th:PiExists} is that $|\Pi_M| = \Omega(n^2)$, where $M = c_1 n^{2/7} \log^{2/7}n$.
The following standard application of the Cauchy-Schwarz inequality shows that the upper bound $\mathcal{Q}_M$ given in Lemma \ref{thm:QK} implies a lower bound on $|\bis|$.

\begin{lemma}\label{thm:applCS}
	$$|\bis| = \Omega\left(n^4|\en_M|^{-1}\right).$$
\end{lemma}
\begin{proof}
	Let
	\[\bis_M = \{\bi{a,b} : (a,b) \in \Pi_M \}.\]
	For a line $\ell$, denote by $w(\ell)$ the number of pairs $(a,b) \in \Pi$ such that $\bi{a,b} = \ell$.
	By the Cauchy-Schwarz inequality,
	\[
	|\en_M| = \sum_{\ell \in \bis_M} w(\ell)^2 \geq \left(\sum_{\ell \in \bis_M} w(\ell) \right)^2|\bis_M|^{-1} = |\Pi|^2|\bis_M|^{-1}.
	\]
	Hence,
	\[
	|\bis| \geq |\bis_M| \geq |\Pi|^2|\en_M|^{-1} = \Omega\left(n^4|\en_M|^{-1}\right).
	\]
\end{proof}

Lemma \ref{thm:QK} gives
$$|\en_M| = O(n^{88/35 + \eps}). $$
Combining this with Lemma \ref{thm:applCS}, we have
$$|\bis| = \Omega(n^{52/35 - \eps}),$$
which is Theorem \ref{thm:distinctBisectors}.

\section{Acknowledgements}
I thank Brandon Hanson, Peter Hajnal, Oliver Roche-Newton, Adam Sheffer, and Frank de Zeeuw for many stimulating conversations on perpendicular bisectors and related questions.
I thank Luca Ghidelli for pointing out an error in Lemma \ref{th:PiExists} in an earlier version.
I thank anonymous referees for numerous helpful comments on the writing and presentation of this paper.

\bibliographystyle{plain}
\bibliography{bisectors}

\end{document}